\documentclass[11pt]{amsart}
\usepackage{amsmath}
\usepackage{amsfonts}
\usepackage{color}
\usepackage{amssymb}
\usepackage{amscd}
\usepackage{amsrefs}
\usepackage{hyperref}
\hypersetup{
    colorlinks=true,
    linkcolor=blue,
    filecolor=magenta,      
    urlcolor=cyan,
}
\newtheorem{thm}{Theorem}[section]

\newtheorem{cor}[thm]{Corollary}

\newtheorem{lemma}[thm]{Lemma}
\newtheorem{prop}[thm]{Proposition}

\theoremstyle{definition}

\newtheorem{remark}[thm]{Remark}
\newtheorem{example}[thm]{Example}

\newcommand{\bb}[1]{\mathbb{#1}}
\newcommand{\cl}[1]{\mathcal{#1}}

\newcommand{\id}{{\operatorname{\textsl{id}\mskip2mu}}}

\newcommand{\OMIN}{{\operatorname{OMIN}}}
\newcommand{\OMAX}{{\operatorname{OMAX}}}

\newcommand{\e}{\varepsilon}

\newcommand{\st}{\ : \ }

\DeclareMathOperator{\tr}{Tr}

\renewcommand{\leq}{\leqslant}
\renewcommand{\geq}{\geqslant}

\newcommand{\M}{\mathsf{M}}

\newcommand{\R}{\mathbf{R}}
\newcommand{\C}{\mathbf{C}}

\newcommand{\N}{\mathbb{N}}

\DeclareMathOperator{\vrad}{\mathrm{vrad}}

\DeclareMathOperator{\E}{\mathbf{E}}

\DeclareMathOperator{\mathspan}{\mathrm{span}}

\DeclareMathOperator{\Ad}{Ad}
\DeclareMathOperator{\CP}{CP}
\DeclareMathOperator{\UCP}{UCP}

\newcommand{\ketbra}[2]{| #1 \rangle \langle  #2 |}

\newcommand{\UkP}{\ifmmode {\textnormal{U}k\textnormal{P}} \else {U$k$P}\fi}

\newcommand{\kmin}{\mathop{\mbox{-}\mathrm{min}}}
\newcommand{\kmax}{\mathop{\mbox{-}\mathrm{max}}}

\allowdisplaybreaks

\begin{document}

\title[CB norms of $k$-positive Maps]{Completely Bounded Norms of $k$-positive Maps}
\author[Aubrun]{Guillaume Aubrun}
\address{Université Lyon 1, CNRS, INRIA, Institut Camille Jordan, 43, bd du 11 novembre 1918, 69100 VILLEURBANNE, France}
\email{aubrun@math.univ-lyon1.fr}
\author[Davidson]{Kenneth R. Davidson}
\address{Department of Pure Mathematics, University of Waterloo, Waterloo, ON, Canada  N2L 3G1}
\email{krdavids@uwaterloo.ca}
\author[M\"uller-Hermes]{Alexander M\"uller-Hermes}
\address{Department of Mathematics, University of Oslo, P.O. box 1053, Blindern,
0316 Oslo, Norway}
\email{muellerh@math.uio.no}
\author[Paulsen]{Vern I.~Paulsen}
\address{Institute for Quantum Computing and Department of Pure Mathematics, University of Waterloo, Waterloo, ON, Canada  N2L 3G1}
\email{vpaulsen@uwaterloo.ca}
\author[Rahaman]{Mizanur Rahaman}
\address{UNIV LYON, INRIA, ENS LYON, UCBL, LIP, F-69342, LYON CEDEX 07, FRANCE}
\email{mizanur.rahaman@ens-lyon.fr}

\begin{abstract}
Given an operator system $\cl S$, we define the parameters $r_k(\cl S)$ (resp.\ $d_k(\cl S)$) defined as the maximal value of the completely bounded norm of a unital $k$-positive map from an arbitrary operator system into $\cl S$ (resp.\ from $\cl S$ into an arbitrary operator system). 

In the case of the matrix algebras $\M_n$, for $1 \leq k \leq n$, we compute the exact value $r_k(\M_n) = \frac{2n-k}{k}$ and show upper and lower bounds on the parameters $d_k(\M_n)$.

Moreover, when $\cl S$ is a finite-dimensional operator system, adapting results of Passer and the 4th author \cite{PaPa}, we show that the sequence $(r_k(\cl S))$ tends to $1$ if and only if $\cl S$ is exact and that the sequence $(d_k(\cl S))$ tends to $1$ if and only if $\cl S$ has the lifting property. 
\end{abstract}

%

\maketitle

\section{Introduction}

Completely positive maps and positive maps play central roles in both operator algebras and quantum information theory. Completely positive maps have a nice representation theory~\cite{Choi75, Choi1972} that does not exist for positive maps.

Let $\cl B(H)$ be the space of continuous linear operators on a Hilbert space $H$.
This is equipped with a norm and a cone of positive (semidefinite) operators.
The algebra $M_n(\cl B(H)):= M_n \otimes B(\cl H)$ of $n\times n$ matrices with coefficients in $\cl B(H)$ can be identified with the space of continuous operators on the direct sum of $n$ copies of $H$.
This provides a norm and a positive cone on matrices over $\cl B(H)$. These are called the {\it matrix norm} and {\it matrix cone} over $B(\cl H)$.

An \textit{operator system} $\cl S$ is any closed self-adjoint subspace of $\cl B(H)$ containing the identity.
An operator system inherits these norms and orders on the matrix spaces $M_n(\cl S)$ via the inclusions $M_k(\cl S) \subseteq M_k(B(\cl H))$.

A linear map $\phi$ from an operator system $\cl S$ to another operator system $\cl T$ is called $k$-positive if the induced map $\id_k \otimes \phi$ of $M_k(\cl S)$ into $M_k(\cl T)$ is positive.
The map is \textit{completely positive} if it is $k$-positive for all $k\ge1$. 
The completely bounded norm of a map between operator systems is given by $\|\phi\|_{cb} = \sup_{k\ge1} \| \id_k \otimes \phi \|$. 
If $\phi$ is completely positive,  then $\| \phi \|_{cb} = \| \phi(I) \|$, where $I$ is the identity. But if $\phi$ is only $k$-positive, then much less is known about $\| \phi \|_{cb},$ even in the case that either the domain or range is a matrix algebra.  


Given two operator systems $\cl S$, $\cl T,$ it is well-known that a unital linear map $\phi: \cl T \to \cl S$ is completely positive if and only if it is {\it completely contractive}, i.e., $\| \phi \|_{cb} \le 1$. Thus, if a unital map is $k$-positive, but not completely positive, then its completely bounded norm must be larger than~$1$. This makes it natural to wonder how large? We are also interested in how the supremum of the completely bounded norms of such maps behaves as $k$ tends to infinity. 
These natural questions have not been addressed even in the case when the underlying operator systems are the matrix algebras $\M_n$. 

Throughout the paper, we use {\UkP} to denote ``unital and $k$-positive''.
In order to state our results, it is convenient to introduce the following parameters associated to an operator system $\cl S$
\[
 r_k(\cl S) := \sup \{ \| \phi \|_{cb} \st \phi: \cl T \to \cl S \,\,\, \UkP \},
\]
 and
 \[
  d_k(\cl S) := \sup \{ \| \phi \|_{cb} \st \phi: \cl S \to \cl T \,\,\, \UkP \},
 \]
 where the supremum ranges over all operator systems $\cl T$ and all {\UkP} maps $\phi$. It is helpful to think as $r$ and $d$ as standing for range and domain. Each sequence $(r_k(\cl S))_k$ and $(d_k(\cl S))_k$ is nonincreasing.
 
In the case that $\cl S$ is the algebra $\M_n$ of $n \times n$ matrices, off-the-shelf information on the parameters $r_k(\M_n)$ reads as
\begin{equation} 
\label{eq:off-the-shelf-domain}
1 = r_n(\M_n) \leq \cdots \leq r_2(\M_n) \leq r_1(\M_n) \leq 2n 
\end{equation}

The equality $r_n(\M_n)=d_n(\M_n)=1$ follows from the fact that a map with domain or range $\M_n$ which is $n$-positive is actually completely positive \cite{Pa2002}*{Theorem 3.14 and Theorem 6.1}. 
The inequality $r_1(\M_n) \leq 2n$ follows from fact that any map $\phi : \cl{T} \to \M_n$ satisfies
\begin{equation} \label{eq:trivial-bound}
 \| \phi \|_{cb}= \| \id_n \otimes \phi: \M_n \otimes \cl T \to \M_n \otimes \M_n \| \leq n \| \phi \|,
\end{equation}
and, if $\phi$ is positive, $\|\phi\| \leq 2 \|\phi(1)\|$. 
Moreover, an early example of Arveson~\cite{arveson69} shows that the constant $2$ is optimal in this inequality.

Our first result is that the rightmost inequality in~(\ref{eq:off-the-shelf-domain}) is not sharp and that the correct value is $r_1(\M_n) = 2n-1$. 
More generally, for any $1 \leq k \leq n$ we show in Theorem \ref{OMINthm} that
\[
 r_k(\M_n) = \frac{2n-k}{k} .
\]
It is essential that we allow for a general operator system in the domain when defining~$r_k(\cl S)$. 
If instead we restrict the domain to be a C*-algebra and take $k=1$, then
\[
 \sup \{ \| \phi \|_{cb} \st \phi: \cl A \to \cl \M_n \text{ unital, positive, }\ \cl A \text{ a C*-algebra} \} = n .
\]
The upper bound follows from \eqref{eq:trivial-bound} combined with the Russo--Dye theorem, 
and is attained by the transpose map $\phi(X)=X^T$ on~$\M_n$. We point out in passing that the transpose is essentially the unique map with this property, see \cite{puzzuoli18}.

The study of the parameter $d_k(\M_n)$ seems to be more intricate. 
In this case, it is easy to see that it suffices to take the range to be $\cl B(H)$ for some separable Hilbert space $H$;
and the supremum is attained. It is also the supremum obtained using range $\M_m$ for arbitrary $m\geq 1$. 

For $k=1$, we show in Theorem \ref{thm:d_1ofM_n} that
\[
d_1(\M_n) = n, 
\]
the equality being achieved by the transpose map. For $k \geq 2$, we show the upper bound
\[
 d_k(\M_n) \leq r_k(\M_n) = \frac{2n-k}{k}.
\]
Essentially no good lower bounds are obtainable from the literature. Generalizing an earlier example by Choi~\cite{Choi1972}, Tomiyama shows in \cite{Tomiyama} that the map $\tau:\M_n \to \M_n$ given by
\begin{equation}\label{equ:TomiyamaEx}
 \tau_{n,k}(X) = \left(1+ \frac{1}{nk-1} \right) \frac{\tr(X)}{n} I_n - \frac{1}{nk-1} X,
\end{equation}
is unital and $k$-positive. Computing the cb norm of this map yields the lower bound 
\[
 d_k(\M_n)\geq \|\tau_{n,k}\|_{cb} = 1 + \frac{2(n-k)}{n(nk-1)}.
\]
A fairly thorough search of the literature \cites{Choi1972, Choi1980, Kye, PaYo, TaTo, Tomiyama} reveals no examples of unital $2$-positive maps with domain $\M_n$ improving on this lower bound. 

Using probabilistic arguments, we obtain in Theorem \ref{theorem:probabilistic-bound} the lower bound
\[ d_k(\M_n) \geq c \sqrt{n/k} \]
for some absolute constant $c>0$.  

One difficulty in the study of maps with domain $\M_n$ is the fact, remarked but not proven in \cite{Pa1986}, that if $n \geq 3$ then there is no $m\in\N$ such that for all $\phi: \M_n \to B(\cl H)$,   $\| \phi \|_{cb} = \| \id_m \otimes \phi: \M_m \otimes \M_n \to \M_m \otimes B(\cl H) \|.$ Since this fact is somewhat central here, we will provide a belated proof in section \ref{S:cb not attained}.

Our work is also motivated by the recent paper of Passer and the 4th author \cite{PaPa} that proved, for finite dimensional operator systems, 
{\it exactness} and the {\it lifting property} could be characterized by the Hausdorff distance between certain sequences of matrix ranges affiliated with the operator systems tending to 0 in the limit. These notions are defined in section~\ref{sec: exactness and lifting prop.}.
We refer the reader to \cite{KPTT} for more information on exactness and lifting properties in the context of operator systems.

The connection between operator systems and matrix ranges uses a choice of a basis for the operator system. 
Our contributions in this direction replaces this vanishing of Hausdorff distance by the parameters $d_k(\cl S)$ and $r_k(\cl S)$, 
which are independent of a basis.  
For an operator system $\cl S$, consider the limits
\[
 r_{\infty} (\cl S) = \lim_{k \to \infty} r_k(\cl S) 
\]
and
\[
 d_{\infty} (\cl S) = \lim_{k \to \infty} d_k(\cl S). 
\]
Building on the work of \cite{PaPa}, we show in Theorem \ref{theorem:infinity} that for finite dimensional operator systems the values of these limits encode structural properties of $\cl S$,
indeed,
\[
 r_{\infty} (\cl S)=1 \iff \cl S \textnormal{ is exact} 
\]
and
\[
 d_{\infty} (\cl S)=1 \iff \cl S \textnormal{ has the lifting property.} 
\]
In section~\ref{sec: exactness and lifting prop.} we show that these results do not extend to infinite dimensional operator systems.

Thus, for finite dimensional operator systems, $r_{\infty}(\cl S)$ and $d_{\infty}(\cl S)$ give measures of non-exactness 
and of the failure of the lifting property.

\section{Background}\label{sec: background}

Recall that an abstract operator system is just a $*$-vector space with a specified order unit and a matrix order satisfying certain axioms.  
Every operator system has a completely order isomorphic representation as a self-adjoint unital subspace of $\cl B(H)$ \cite{CE} 
(see \cite{Pa2002}*{Theorem 13.1}) and the matrix order endows the operator system with a matrix norm, making it into an operator space.
Given an operator system $\cl S$, we let  $\M_n(\cl S)$ denote the vector space of $n \times n$ matrices with entries from $\cl S$ and let $\M_n(\cl S)^+ \subseteq \M_n(\cl S)$ denote the cone of positive elements.
Given an operator system and a natural number $k$, Xhabli \cites{Xh1, Xh2} introduced new operator systems, denoted $\OMAX^k(\cl S)$ and $\OMIN^k(\cl S)$ (see also \cite{PTT}) . 

We note that since Xhabli's work appeared, the superscript has often been replaced by a subscript by subsequent authors, i.e., $\OMAX_k(\cl S) \equiv \OMAX^k(\cl S)$ and $\OMIN_k(\cl S) \equiv \OMIN^k(\cl S)$. In particular, this is the case in the work of A.S. Kavruk \cite{Ka} and \cite{PaPa}, which we will be referencing, so we adopt the subscript notation here as well.

As unital $*$-vector spaces, $\OMIN_k(\cl S)$ and $\OMAX_k(\cl S)$ are just $\cl S$ with different matrix orders.  
We need to name several maps between these spaces which are setwise just the identity map.
If $\cl A$ and $\cl B$ are two operator spaces or systems which are the same set with different matricial norms or orders, then we will write $\id_{\!\cl A}^{\, \cl B}$ to
denote $\id:\cl A \to \cl B$.
The new matrix orders are characterized uniquely by the following universal properties:
\begin{enumerate}
\item $\id_{\OMAX_k(\cl S)}^{\ \cl S}: \OMAX_k(\cl S) \to \cl S$ and $\id_{\cl S}^{\,\OMIN_k(\cl S)}: \cl S \to \OMIN_k(\cl S)$ are unital completely positive (UCP),
\item $\id_{\cl S}^{\,\OMAX_k(\cl S)}: \cl S \to \OMAX_k(\cl S)$ and $\id_{\OMIN_k(\cl S)}^{\ \cl S}: \OMIN_k(\cl S) \to \cl S$ are unital and $k$-positive (\UkP),
\item for any other operator system $\cl T$, $\phi: \cl T \to \cl S$ is {\UkP} if and only if $\id_{\cl S}^{\,\OMIN_k(\cl S)} \circ \phi: \cl T \to \OMIN_k(\cl S)$ is UCP,
\item for any other operator system $\cl T$, $\psi: \cl S \to \cl T$ is {\UkP} if and only if $\psi \circ \id_{\OMAX_k(\cl S)}^{\ \cl S}: \OMAX_k(\cl S) \to \cl T$ is UCP.
\end{enumerate}

We let $\M_n(\OMAX_k(\cl S))^+ , \M_n(\OMIN_k(\cl S))^+ $ denote the matrix cones of $\OMAX_k(\cl S)$ and $\OMIN_k(\cl S)$ respectively. Note that it holds that
\[\M_n(\OMAX_k(\cl S))^+ \subseteq \M_n(\cl S)^+; \ \text{and} \ \ \M_n(\cl S)^+\subseteq\M_n(\OMIN_k(\cl S))^+ .\]

In the case that $\cl S= \M_m$ these cones appear in many other places in the literature.  In \cite{JoKrPaPe}*{Theorem~5} it is shown that
 $\rho \in \M_n(\OMAX_k(\M_m))^+$ if and only if $SN(\rho) \leq k$, while $\rho \in \M_n(\OMIN_k(\M_m))^+$ if and only if $\rho$ is $k$-block positive. In particular, we have $\OMAX_m(\M_m)=\M_m=\OMIN_m(\M_m)$.
 
 Explicit descriptions of these cones for a general operator system can be found in Xhabli's papers, and in \cite{Ka} and \cite{PaPa}, so we do not repeat them here.  We do remark on one fact that we shall use.  Given an operator system $\cl S$, the set 
 \[
  \Omega_k := \UCP(\cl S, \M_k) =  \{ \phi \, \vert \, \phi: \cl S \to \M_k \, \textnormal{UCP} \},
 \]
 equipped with the weak$*$-topology, is often called the \emph{$k$-th matrix state space}. Given $x \in \cl S$ we let
 \[
  \widehat{x}: \Omega_k \to \M_k,
 \]
 be the continuous affine function given by $\widehat{x}(\phi) = \phi(x)$. Thus, $\widehat{x}$ is an element of the C*-algebra $C(\Omega_k) \otimes \M_k$, which we identify with the continuous functions from $\Omega_k$ to $\M_k$.
With these identifications, the map
\[
 \Gamma: \OMIN_k(\cl S) \to C(\Omega_k) \otimes \M_k; \,\, x \to \widehat{x},
\]
is a unital complete order embedding. So we may identify
\[
 \OMIN_k(\cl S) \equiv \{ \widehat{x}: x \in \cl S \} \subseteq C(\Omega_k) \otimes \M_k.
\]
 
 These operator system constructions give us another way to study the parameters that we introduced earlier.
 
 \begin{prop} \label{prop:rkdkViaId}
 Let $\cl S$ be an operator system. Then
 \[
  r_k(\cl S) =  \| \id_{\OMIN_k(\cl S)}^{\ \cl S} \|_{cb},
 \]
 and
 \[
  d_k(\cl S)  = \| \id_{\cl S}^{\,\OMAX_k(\cl S)} \|_{cb}.
 \]
 \end{prop}

 
 \begin{proof} 
 We only prove the first equality, the second is similar.  
 Since the map $\id_{\OMIN_k(\cl S)}^{\ \cl S} $ is \UkP, the left hand side is larger than the right hand side. 
 Given a {\UkP} map $\phi: \cl T \to \cl S$, let $\psi = \id_{\cl S}^{\,\OMIN_k(\cl S)} \circ \phi$.
 By the properties of $\OMIN_k(\cl S)$, $\psi$ is CP with $\|\psi\|_{cb} = \|\psi(1)\| = 1$.
 Now $\phi = \id_{\OMIN_k(\cl S)}^{\ \cl S} \circ \psi$; whence 
 \[
  \| \phi \|_{cb} \leq \| \id_{\OMIN_k(\cl S)}^{\ \cl S} \|_{cb} \|\psi \|_{cb} =  \| \id_{\OMIN_k(\cl S)}^{\ \cl S} \|_{cb} .
 \]
Therefore $r_k(\cl S) = \| \id_{\OMIN_k(\cl S)}^{\ \cl S} \|_{cb}$.
 \end{proof}


 \section{Bounds for matrix algebras}

Recall that a $k$-positive map with domain or range equal to the matrix algebra $\M_n$ is completely positive for any $k \geq n$. In this section we study $r_k(\M_n)$ and $d_k(\M_n)$ for $k<n$, and we will need the following norm: Given two operator systems $\cl S, \cl T$ and a linear map $\phi: \cl S \to \cl T$ such that $\phi(X^*) = \phi(X)^*$, 
 Haagerup \cite{Hag} defines the \emph{decomposition norm} of the map to be
 \begin{equation}\label{equ:DecNorm}
  \| \phi \|_{dec} = \inf  \big\{ \|\phi^+(I) + \phi^-(I) \| : \phi^\pm \in \CP(\cl S, \cl T), \ \phi = \phi^+-\phi^- \big\},
 \end{equation}
 and $\| \phi \|_{dec} = + \infty$ if no such decomposition exists.
 It is easy to see that $\| \phi \|_{cb} \leq \| \phi \|_{dec}$. 
 By Wittstock's decomposition theorem (\cite{Wi}, see also chapter 8 of \cite{Pa2002}), when $\cl T$ is an injective operator system such as $\M_n$, then $\| \phi \|_{cb} = \| \phi \|_{dec}$.

Let $\cl U_n$ denote the unitary group of $\M_n$.
The map $\Ad_U(X) = U^*XU$ is a completely isometric automorphism.
We say that a map $\phi : \M_n \to \M_n$ is \emph{covariant} if 
$\phi(\Ad_U(X)) = \Ad_U ( \phi(X) )$ for every $U \in \cl U_n$ and $X \in \M_n$. 
We need the following lemma. 


 \begin{lemma} \label{lemma:wittstock-covariant}
 For any self-adjoint covariant map $\phi : \OMIN_k(\M_n)\to \M_n$ we have
 \[
  \|\phi\|_{dec} = \min \big\{ \|\phi^+(I_n) + \phi^-(I_n) \| : \phi^\pm \in \CP(\OMIN_k(\M_n), \M_n), \ \phi = \phi^+-\phi^- \big\},
 \]
 where the minimum is achieved by covariant maps $\phi^+$ and $\phi^-$.
 \end{lemma}

 \begin{proof}
 Clearly, the infimum in \eqref{equ:DecNorm} is attained when the range is $M_n$. Note that for any $U \in \M_n$ the map $\gamma_U: \OMIN_k(\M_n) \to \OMIN_k(\M_n)$ given by $\gamma_U(X) = U^*XU$ is completely positive.  
 Indeed, we have 
 \[
 \gamma_U = \id_{\M_n}^{\,\OMIN_k(\M_n)} \circ \Ad_U \circ \id_{\OMIN_k(\M_n)}^{\,\M_n}
 \]
 Now $\Ad_U \circ \id_{\OMIN_k(\M_n)}^{\,\M_n}$ is $\UkP$ since $\Ad_U$ is UCP, and hence by the property of $\OMIN_k(\M_n)$, $\gamma_U$ is UCP. Assume that $\phi$ is covariant and let $\phi^{\pm}$ be maps achieving the minimum. We set
 \[
  \psi^{\pm} = \int_{\cl U_n} \Ad_U\circ \phi^{\pm}\circ \gamma_U \, dU ,
 \]
 where the integration is with respect to the Haar measure on the compact group $\cl U_n$. Observe that the maps $\psi^{\pm}: \OMIN_k(\M_n) \to \M_n$ are CP. Since $\phi$ is covariant, we have
 \begin{align*}
  \phi(X) &= \int_{\cl U_n} U \phi(U^*XU) U^* \, dU \\&
  = \int_{\cl U_n} U \phi^+(U^*XU) U^* \, dU - \int_{\cl U_n} U \phi^-(U^*XU) U^* \, dU \\&
  = \psi^+(X) - \psi^-(X). 
 \end{align*}
 Using the triangle inequality and the minimality of $\phi^\pm$, we have
 \[
  \|\phi^+(I_n) + \phi^-(I_n) \| \leq \|\psi^+(I_n) + \psi^-(I_n) \| \leq \|\phi^+(I_n) + \phi^-(I_n) \| .
 \]
 Therefore, equality holds throughout and the result follows since the maps $\psi^\pm$ are covariant.
 \end{proof}

The next result is the main theorem of this section.

\begin{thm}\label{OMINthm} For any $k, n \in \bb N$ with $k <n$ we have that
\[
 r_k(\M_n) = \| \id_{\OMIN_k(\M_n)}^{\,\M_n} \|_{dec} = \frac{2n-k}{k}.
\]
\end{thm}

\begin{proof} 
By Proposition \ref{prop:rkdkViaId} we have 
\[
 r_k(\M_n) = \| \id_{\OMIN_k(\M_n)}^{\,\M_n} \|_{cb} = \| \id_{\OMIN_k(\M_n)}^{\,\M_n} \|_{dec} ,
\]
where the second equality follows from Wittstock's decomposition theorem (\cite{Wi}, see also chapter 8 of \cite{Pa2002}) since $\M_n$ is injective.

In order to apply Lemma \ref{lemma:wittstock-covariant}, we note that $*$-preserving covariant maps on $\M_n$ are parametrized as
\[ 
\phi_{s,t} (X) = s X + t \frac{ \tr X }{n} I_n
\]
for real numbers $s$ and $t$, see \cite{Kye}*{Proposition~1.7.2}. By Proposition~\ref{k-PEBprop}, the set of covariant CP maps from $\OMIN_k(\M_n)$ into $\M_n$ coincides with the set of covariant $k$-PEB maps. By \cite{Kye}*{Theorem~1.7.3} (which covers the case $s+t=1$) the parameters $(s,t)\in \R^2$ corresponding to these maps are given by
\begin{align*}
\Gamma &= \{ (s,t) \in \R^2 \st \phi_{s,t} \textnormal{ is }k\textnormal{-PEB} \} \\
&=\left\{ (s,t) \st s+t \geq 0 \ \textnormal{ and } -\frac{s+t}{n^2-1} \leq s \leq (s+t)\frac{nk-1}{n^2-1} \right\}.
\end{align*}

Applying Lemma \ref{lemma:wittstock-covariant} to the covariant map $\id_{\OMIN_k(\M_n)}^{\,\M_n}$ shows that
\begin{align*}
 r_k(\M_n) &= \| \id_{\OMIN_k(\M_n)}^{\,\M_n} \|_{dec} \\&=  
 \min \big\{ \|\phi_{s+1,t}(I_n)+\phi_{s,t}(I_n) \| \st (s,t) \in \Gamma, (s+1,t) \in \Gamma \big\}. 
\end{align*}
We compute that $\|\phi_{s+1,t}(I_n)+\phi_{s,t}(I_n) \| = 2s+2t+1$. 
Reparameterizing, we set $r=s+t$, $a= \frac{1}{n^2-1}$ and $b =\frac{nk-1}{n^2-1}$. 
We now need to minimize the quantity $2r+1$ under the constraints:
\begin{enumerate}
\item $r \geq 0$,
\item $ -ar \leq s \leq br$,
\item $-a(r+1) \leq s+1 \leq b(r+1)$.
\end{enumerate}
Thus, $s \leq br$ and $s+1 \leq b(r+1)$ imply that
\[
  s \leq br + (b-1),
\]
and $-ar \leq s$ and $-a(r+1) \leq s+1$ imply that
\[ 
 -ar \leq s.
\]
So the smallest value of $r$ is where the two boundary lines intersect, at 
\[
 r= \frac{1-b}{a+b}=\frac{1- \frac{nk-1}{n^2-1}}{\frac{1}{n^2-1} + \frac{nk-1}{n^2-1}} =\frac{n^2-nk}{nk} = \frac{n-k}{k}.
\]
This yields a value for the minimum of $2r+1 = \frac{2n-k}{k}$.
\end{proof}
We get an immediate corollary for unital $k$-positive maps between matrix algebras having $\M_n$ as range:
\begin{cor}
Let $\phi: \M_m\rightarrow \M_n$ be a unital and $k$-positive map. Then 
\[\|\phi\|_{cb}\leq \frac{2n-k}{k}.\]
\end{cor}

The following rephrasing of Theorem \ref{OMINthm} might be of independent interest:

\begin{cor} For $X\in \M_m(\M_n)$ we have
\[
 \|X \| \leq \frac{2n-k}{k} \max \{ \| (I_m\otimes P)(X\otimes I_k)(I_m\otimes P) \| : P \in \M_{nk} \text{ a rank $k$ projection} \},
\]
and there exists a non-zero $X\in \M_n(\M_n)$ achieving equality in this bound.
\end{cor}

\begin{proof} 
We consider $X$ as an element of $\M_m(\OMIN_k(\M_n))$.
Recall the representation of $\OMIN_k(\M_n)$ as a subset of $C(\Omega_k; \M_k)$, where $\Omega_k = \UCP(\M_n, \M_k)$.
Every UCP map from $\M_n$ into $\M_k$ has the form $\phi(X) = V^*(X \otimes I_k)V$ for some isometry $V: \bb C^k \to \bb C^{nk}$.
Therefore taking $P = VV^*$, an arbitrary projection of rank $k$ in $\M_{nk}$,
\begin{align*}
 \|X\|_{\M_m(\OMIN_k(\M_n))} &= \sup_{\phi \in \UCP(\M_n, \M_k)} \| \phi_m(X) \|
 \\
 &= \sup_{\substack{P = P^2 = P^* \\ P \in \M_{nk}}}  \| (I_m\otimes P)(X\otimes I_k)(I_m\otimes P) \| .
\end{align*}
Therefore, applying $(\id_{\OMIN_k(\M_n)}^{\,\M_n})_m$ to $X$ and using Theorem~\ref{OMINthm}, we obtain
\[
 \|X\| \leq \frac{2n-k}k \sup_{\substack{P = P^2 = P^* \\ P \in \M_{nk}}}  \| (I_m\otimes P)(X\otimes I_k)(I_m\otimes P) \| .
\]

The second statement follows from the fact that the cb-norm of any map $\psi$ with range $\M_n$ is equal to $\| \id_n \otimes \psi \|$, see \cite{Pa2002}.
\end{proof}

Theorem ~\ref{OMINthm} yields an upper bound for $d_k(\M_n)$ as well.

\begin{cor} \label{cor:dk}
For any $k,n \in \bb N$ we have that
\[
 d_k(\M_n) \leq\| \id_{\M_n}^{\,\OMAX_k(\M_n)} \|_{dec} = \frac{2n-k}{k}.
\]
\end{cor}

\begin{proof} 
By Proposition \ref{prop:rkdkViaId} and the definition of the decomposition norm, we have
\[
 d_k(\M_n) = \| \id_{\M_n}^{\,\OMAX_k(\M_n)} \|_{cb}\leq \| \id_{\M_n}^{\,\OMAX_k(\M_n)} \|_{dec} . 
\]
By Proposition~\ref{k-PEBprop}, the CP maps from $\M_n$ into $\OMAX_k(\M_n)$ coincide with the CP maps from $\OMIN_k(\M_n)$ into $M_n$ and hence we have
\[
\| \id_{\M_n}^{\,\OMAX_k(\M_n)} \|_{dec} = \| \id_{\OMIN_k(\M_n)}^{\,\M_n} \|_{dec} = \frac{2n-k}{k},
\]
where the final equality follows from Theorem~\ref{OMINthm}.
 \end{proof}
We get an immediate corollary for unital $k$-positive maps between matrix algebras having $\M_n$ as domain:

\begin{cor}
Let $\cl T$ be an operator system and let $\phi: \M_n\rightarrow \cl T$ be a unital and $k$-positive map. Then 
\[\|\phi\|_{cb}\leq \frac{2n-k}{k}.\]
\end{cor}

It is a natural question whether $d_k(\M_n) = \frac{2n-k}{k}$ holds in general. We can show that for $k=1$ this is not the case:

\begin{thm}\label{thm:d_1ofM_n}
For every $n\geq 2$, we have $d_1(\M_n)=n$.
\end{thm} 

\begin{proof}
Note that $d_1(\M_n)\geq n$ since the transpose map $\phi:\M_n\rightarrow \M_n$ given by $\phi(X)=X^T$ in any fixed basis is positive and satisfies $\|\phi\|_{cb}=n$. For the other direction, let $\cl T$ be any operator system, and $\phi : \M_n \to \cl T$ be a unital positive map. Fix an integer $m \geq 1$ and an element $x \in \M_m(\M_n)$ such that $\|x\| \leq 1$. By \cite{ando}*{Corollary~8.4}, the operator
\[
\begin{pmatrix} I_m\otimes I_n & \frac{1}{n}x \\
\frac{1}{n}x^* & I_m\otimes I_n\end{pmatrix} \in \M_{2m}(\M_n)
\]
is separable, i.e., is a positive combination of elements of the form $y \otimes z$ for $y \in \M_{2m}^+$ and $z \in \M_n^+$. Since $\phi$ is unital and positive, the operator
\[ (\id_{2m} \otimes \phi) \begin{pmatrix} I_m\otimes I_n & \frac{1}{n}x \\
\frac{1}{n}x^* & I_m\otimes I_n\end{pmatrix} = 
\begin{pmatrix} I_m\otimes I_{\cl T} & \frac{1}{n} \phi(x) \\
\frac{1}{n} \phi(x)^* & I_m\otimes I_{\cl T}\end{pmatrix}
\]
is positive, where $I_{\cl T}$ is the identity element in $\cl T$.
Thus it follows that $\|\phi(x)\| \leq n$, and hence we have  $\|\phi\|_{cb} \leq n$ and $d_1(\M_n) \leq n$.
\end{proof}

%

The previous theorem shows that the equality $d_k(\M_n) = \frac{2n-k}{k}$ does not always hold. From the proof of Corollary \ref{cor:dk} it holds if and only if 
 \[ 
  \| \id_{\M_n}^{\, \OMAX_k(\M_n)} \|_{cb} = \| \id_{\M_n}^{\, \OMAX_k(\M_n)} \|_{dec}, 
 \]
 and hence for particular values of $n$ and $k$ the equality would follow if $\OMAX_k(\M_n)$ was an injective operator system. Unfortunately, this is never the case when $k<n$:

\begin{prop} 
For any $n,k \in \bb N$ with $k <n$,  $\OMAX_k(\M_n)$ is not an injective operator system.
\end{prop}

 \begin{proof} 
 By a result of Choi--Effros \cite{CE}, every injective operator system is unitally completely order isomorphic to a unital C*-algebra.  
 Every finite dimensional C*-algebra is a direct sum of matrix algebras. 
Thus, up to unital complete order isomorphism, we would have that $\OMAX_k(\M_n) = \oplus_j \M_{n_j}.$
 
Since every $k$-positive map with domain $\OMAX_k(\M_n)$ is CP, we require that for each summand, $n_j \leq k$.  
But then since $\id: \M_n \to \OMAX_k(\M_n) = \oplus_j \M_{n_j}$ is $k$-positive, we would have that it is CP.   
This implies that the identity map is a unital complete order isomorphism between $\M_n$ and $\OMAX_k(\M_n)$, 
which in turn implies that every $k$-positive map with domain $\M_n$ is completely positive. However, in the literature there are many examples of $k$-positive maps with domain $\M_n$ that are not CP; for example the map in \eqref{equ:TomiyamaEx}. This contradiction completes the proof. 
 \end{proof}

\section{A probabilistic lower bound on \texorpdfstring{$d_k(\M_n)$}{dk(Mn)}}

We start this section by noting that results about maps $\Phi : \M_m \to \M_n$ between matrix algebras have an immediate dual translation involving the adjoint map $\Phi^* : \M_n \to \M_m$, 
with respect to the Hilbert-Schmidt inner product
\[\langle A, B\rangle =\tr(A^*B),\]
for $A, B$ in $\M_m$ or in $\M_n$. This translation is obtained using the equivalences
\begin{gather*}
 \Phi \textnormal{ is } k\textnormal{-positive} \iff \Phi^* \textnormal{ is } k\textnormal{-positive}, \\
 \Phi \textnormal{ is unital} \iff \Phi^* \textnormal{ is trace-preserving}, \\
 \|\Phi\|_{cb} = \|\Phi^*\|_{\diamond} , 
\end{gather*}
where $\|\cdot\|_{\diamond}$ is the completely bounded trace norm, often called the diamond norm in quantum information, and is defined as follows for any linear map $\Psi$ acting on $\M_n$:
\[\|\Psi\|_{\diamond}=\sup_{k\geq 1} \ \sup \{\|(\id_k\otimes \Psi)(X)\|_1: X\in \M_k\otimes \M_n, \|X\|_1\leq 1\}. \]

We now turn our attention to obtaining lower bounds on $d_k(\M_n)$.  Surprisingly, it seems to be difficult to find concrete examples of unital $k$-positive maps for $k \geq 2$ on a matrix algebra with large cb-norm. A review of the literature yields no examples of such a map with cb-norm larger than 2. To obtain a better lower bound, we use the probabilistic method. We prove the following theorem.

\begin{thm} \label{theorem:probabilistic-bound}
There exists a constant $c>0$ with the following property: for every integer $n \geq 1$ and $1 \leq k \leq n$, 
there is a $k$-positive unital map $\Phi : \M_n \to \M_n$ such that $\|\Phi\|_{cb} \geq c \sqrt{n/k}$. 
In particular, we have
\[
 d_k(\M_n) \geq c \sqrt{n/k}. 
\]
\end{thm}

The proof uses the probabilistic method through the concept of Gaussian mean width. 
Let $E$ be a finite-dimensional real Euclidean space. A standard Gaussian vector in $E$ is a random variable $\Gamma$, 
taking values in $E$ and such that, for every orthonormal basis $(e_1,\dots,e_n)$ of $E$, 
the random variables $(\langle \Gamma,e_i \rangle)_i$ are i.i.d.\ with an $N(0,1)$ distribution.
If $K \subset E$ is a bounded subset, we define its \emph{Gaussian mean width} as $w_G(K) = \E \sup_{x \in K} \langle \Gamma, x \rangle$. 
The Gaussian mean width is intrinsic: if $F$ is a subspace of $E$ and $K \subset F$, 
we may equivalently compute $w_G(K)$ as a subset of $E$ or as a subset of $F$.

The real vector space $\M_n^{sa}$ of $n \times n$ self-adjoint matrices with complex entries is equipped with the Euclidean structure induced by the Hilbert--Schmidt inner product $\langle A,B \rangle = \tr A^*B$.

\begin{proof}
We show the dual statement: there is a $k$-positive trace-preserving map $\Psi : \M_n \to \M_n$ such that $\|\Psi\|_{\diamond} \geq c \sqrt{n/k}$. 
The theorem then follows by taking $\Phi = \Psi^*$.

Consider the set
\[
 P_{n,k} = \{ \Psi : \M_n \to \M_n \ k\textnormal{-positive and trace-preserving} \} .
\]
Also let $B_\diamond$ be the unit ball for the normed space  $(\cl B(\M_n^{sa}),\| \cdot\|_\diamond)$. 
We claim that, for some constants $c_1>0$ 
\begin{equation} \label{eq1}
 w_G(P_{n,k}) \geq c_1 n^{5/2}/\sqrt{k} 
\end{equation}
and
\begin{equation} \label{eq2}
 w_G(B_\diamond)) \leq 2 n^2. 
\end{equation}
The conclusion is now easy: if $\lambda$ denotes the smallest positive number such that $P_{n,k} \subset \lambda B_\diamond$ 
(note that these sets have different dimensions), then $w_G(P_{n,k}) \leq w_G(\lambda B_\diamond) = \lambda w_G(B_\diamond)$ 
and therefore $\lambda \geq c \sqrt{n/k}$ for $c=c_1/2$. So
\[
 \sup_{\Psi \in P_{n,k}} \|\Psi\|_{\diamond} \geq c \sqrt{n/k} .
\]

It remains to prove \eqref{eq1} and \eqref{eq2}. The inequality \eqref{eq1} follows from the results of \cite{SWZ10}. A sketch of the estimate $\vrad(P_{n,k}) = \Theta (\sqrt{n/k})$ appears in \cite{SWZ10}*{paragraph 4.3} (here $\vrad$ denotes the volume radius). The estimate~\eqref{eq1} then follows from the inequality $\vrad (K) \lesssim w_G(K)/ \sqrt{N}$ valid for any $N$-dimensional set (see \cite[p.\ 95]{AubrunSzarek17}). Here we have $N=n^4-n^2$.

To prove \eqref{eq2}, we introduce the unnormalized maximally entangled state
$\chi = \sum_{i=1}^n e_i \otimes e_i \in \C^n \otimes \C^n$ built on the canonical basis $(e_i)$, so that the rank one operator
\[ \ketbra{\chi}{\chi} = \sum_{i,j} E_{i,j} \otimes E_{i,j},\]
where $E_{i,j}$ are the canonical matrix units. The Choi--Jamio{\l}kowski isomorphism identifies linear maps on $\M_n$ with matrices in $\M_{n^2} = \M_n \otimes \M_n$ via
$\alpha : \cl B(\M_n) \to \M_{n^2}$ defined as 
\[
 \alpha(\Phi) = (\id \otimes \Phi)(\ketbra{\chi}{\chi}) = \big[ \Phi(E_{ij}) \big]
\]
and has the property that $\Phi$ is completely positive if and only if $\alpha(\Phi)$ is a positive semidefinite matrix~\cite{Choi75}.
For any $\Phi : \M_n \to \M_n$, we have 
$\|\Phi\|_\diamond \geq \frac{1}{n} \| \alpha(\Phi) \|_1$ and thus
$\alpha(B_\diamond) \subset n B_1$, where $B_1$ is the unit ball of the normed space $(\M_{n^2}^{sa},\|\cdot\|_1)$. 
Since $\alpha$ is an isometry between the underlying Euclidean spaces, we have
\[
 w_G(B_\diamond) = w_G(\alpha(B_\diamond)) \leq n\, w_G(B_1).
\]
Since the trace norm is dual to the operator norm, we have $w_G(B_1) = \E \|\Gamma\|$ 
where $\Gamma$ is a standard Gaussian vector in $\M_{n^2}^{sa}$, i.e., a GUE matrix. 
We conclude by using the fact that a $p \times p$ GUE matrix $\Gamma$ satisfies the inequality 
$\E \|\Gamma\| \leq 2\sqrt{p}$ (see \cite{Szarek05}*{Appendix F}).
\end{proof}

\begin{remark}
A numerical estimate for the constant $c$ appearing in Theorem~\ref{theorem:probabilistic-bound}, of the order $10^{-2}$, can be obtained by tracking the argument from \cite{SWZ10}. The main loss comes from the discretization argument used in \cite[Section~4.1]{SWZ10}.
\end{remark}

\section{Maps with Domain \texorpdfstring{$\M_n$}{Mn} where  the Completely Bounded Norm is not attained} \label{S:cb not attained}

For maps $\phi: X \to \M_n$ with $X$ an operator space, a result of R.~Smith \cite{smith} proves that $\|\phi\|_{cb} = \| \id_n \otimes \phi\|$. 
 In \cite{Pa1986} and \cite{Pa2002}, it is commented without proof that a result of Haagerup's implies that there is no corresponding result for completely bounded maps with domain $\M_n$ (see also \cite{PuWa}).  Since we believe that this phenomenon is related to the difficulty of computing $d_k(\M_n)$, we expand upon this comment here and provide a proof.

Let $\cl X \subseteq \cl A$ be an operator space, where $\cl A$ is a unital C*-algebra 
and consider the by now fairly standard operator system affliated with $\cl X$,
\[
 \cl S_{\cl X} = \Big\{ \begin{pmatrix} aI_{\cl A} &X \\Y^* & b I_{\cl A} \end{pmatrix} : X, Y \in \cl X, \,\, a,b \in \bb C \Big\} \subseteq \M_2(\cl A).
\]

\begin{lemma} 
A linear map $\phi: \cl X \to \cl B(H)$ is $m$-contractive if and only if $\Phi: \cl S_{\cl X} \to \cl B(H \oplus H)$ defined by
\[ 
 \Phi \big( \begin{pmatrix} aI_{\cl A} & X \\ Y^* & bI_{\cl A} \end{pmatrix} \big) 
 = \begin{pmatrix} aI_{\cl H} & \phi(X) \\ \phi(Y)^* & b I_{\cl H} \end{pmatrix} 
\]
is m-positive.
\end{lemma}

This follows from the same proof as given in \cite{Pa1986}*{Lemma~7.1} and \cite{Pa2002}*{Lemma~8.1}.

\begin{thm} \label {T:lifting mpos}
Let $\cl X$ be a finite dimensional operator space, let $\cl B$ be a unital C*-algebra, let $ \cl J \subseteq \cl B$ be a two-sided ideal, 
let $\phi: \cl X \to \cl B/\cl J$ and let $m \geq 1$.
Then there exists a lifting $\psi: \cl X \to \cl B$ with $\| \id_m \otimes \psi \| = \|  \id_m \otimes \phi \|$;
i.e., $\phi = q \circ \psi$ where $q : \cl B \to \cl B/\cl J$ is the quotient map.
\end{thm}

\begin{proof} 
Without loss of generality we may assume that $\|\phi \otimes \id_m \|= 1$, 
so that $\Phi: \cl S_X \to \M_2(\cl B) /\M_2(\cl J) = \M_2(\cl B/\cl J)$ is unital and $m$-positive.
By the result of A. G. Robertson and R. R. Smith \cite{RS89}*{Proposition~2.4} or by Kavruk's result \cite{Ka}*{Corollary 9.12}, $\Phi$ has a unital $m$-positive lifting, $\Psi: \cl S_X \to \M_2(\cl B)$.
The (1,2)-corner of this map is the desired $\psi$.
\end{proof}

\begin{cor} \label{cor:lifting}
Let $\cl X$ be a finite dimensional operator space, and let $\cl B$ be a unital C*-algebra and $\cl J$ a two-sided ideal.
Suppose that there exists $m \geq 1$ so that for every $\psi: \cl X \to \cl B$, we have $\|\psi\|_{cb} = \|  \id_m \otimes \psi \|$. 
Then every map $\phi: \cl X \to \cl B /\cl J$ has a lifting $\psi: \cl X \to \cl B$ with the same cb-norm. 
\end{cor}

\begin{proof} 
Given $\phi: \cl X \to \cl B /\cl J$, use Theorem~\ref{T:lifting mpos} to find a lifting $\psi:\cl X \to \cl B$ with $\|  \id_m \otimes \psi \| = \|  \id_m \otimes \phi \|$.
Then 
\[
 \| \psi \|_{cb} = \|  \id_m \otimes \psi \| = \|  \id_m \otimes \phi \| \leq \|\phi\|_{cb} \leq \| \psi \|_{cb}.
\]
Therefore $\psi$ is the desired lifting.
\end{proof}

The following is a restatement of a result of Haagerup \cite{Hag}. For those familiar with the concepts, $\cl R$ denotes the hyperfinite $II_1$-factor and $\cl R^{\omega}$ is an ultrapower. For those unfamiliar with these concepts, it is enough to remark that $\cl R^{\omega}$ is a unital C*-algebra and it is a quotient of the C*-algebra $\ell^{\infty}(\bb N, \cl R)$.

\begin{cor}[Haagerup] 
Let $n>2$, then there is no constant $m \in \bb N$ such that for every linear map $\psi: \ell^{\infty}_n \to \ell^{\infty}(\bb N, \cl R)$ one has 
\[ 
\| \id_m \otimes \psi \| = \| \psi \|_{cb}.
\] 
\end{cor}

\begin{proof}
In \cite{Hag}*{Example 3.1}, U.~Haagerup gives an example of a completely contractive map with domain $\ell^{\infty}_n, \, n > 2$, 
into $\cl R^\omega = l^\infty(\bb N, \cl R)/I_\omega$ 
with no completely contractive lifting to $l^\infty(\bb N, \cl R)$.  
In fact, any lifting has cb-norm at least $\frac{n}{2\sqrt{n-1}}$.

Thus, by Corollary~\ref{cor:lifting}, there does not exist an $m$ where the cb-norm is attained for all maps into $\ell^{\infty}(\bb N, \cl R)$.
\end{proof}

By projecting $\M_n$ onto its diagonal one obtains the same result with domain $\M_n$ for $n >2$. 
For maps with domain $\ell^{\infty}_2$, the norm and cb-norm are equal for every map, see for example \cite{Hag}*{Proposition 3.4}. 
So $n > 2$ is essential in Haagerup's result.

A perhaps interesting historical note. Upon giving the example of the completely contractive map with no completely contractive lifting, \cite{Hag} remarks that, 
``this gives an answer to a question raised by Paulsen in his 1983 talk at the AMS meeting in Denver", without specifying the question. 
Since the talk is not available, we remark that the question was whether or not there exists a constant $m$, such that for all maps $\psi: \M_n \to B(\cl H)$ 
one has $\| \psi \otimes \id_m \| = \| \psi \|_{cb}$? 
If such a constant $m$ existed for $B(\cl H)$ the same constant would work for maps into any operator space, and so we know that there is no such $m$.

\begin{remark} Given an operator space $\cl X$, there are operator spaces $\textrm{MIN}_k(\cl X)$ and $\textrm{MAX}_k(\cl X)$ that are the analogues of our operator systems $\OMIN_k(\cl S)$ and $\OMAX_k(\cl S)$.  In particular, every $k$-contractive map from $\cl X$ to another operator space $\cl Y$ is completely contractive as a map from $\textrm{MAX}_k(\cl X)$ to $\cl Y$. By Theorem~\ref{T:lifting mpos} every $k$-contractive map from $\cl X$ to a quotient $\cl B/\cl J$ has a completely contractive lifting from $\textrm{MAX}_k(\cl X)$ to $\cl B$. If we consider the identity map $\gamma_k: \ell^{\infty}_n \to \textrm{MAX}_k(\ell^{\infty}_n)$, Haagerup's result  shows that $\| \gamma_k \|_{cb} \geq \frac{n}{2 \sqrt{n-1}}$, for every $k \in \bb N$.
In contrast, since $\ell^{\infty}_n$ is an abelian C*-algebra, every positive map with domain or range $\ell^{\infty}_n$ is completely positive and hence, $\ell^{\infty}_n = \OMIN_k(\ell^{\infty}_n) = \OMAX_k(\ell^{\infty}_n)$ for every $k$.
\end{remark}

\section{Exactness, Lifting Properties and Matrix Ranges}\label{sec: exactness and lifting prop.}

In \cite{PaPa} the authors studied finite dimensional operator systems of the form 
\[\cl S_{\rm T}= \mathspan \{ I, T_1,...,T_d, T_1^*,...,T_d^* \}\] and characterized which operator systems 
had exactness and lifting properties in terms of a certain Hausdorff distance involving the joint matrix ranges of the $d$-tuple $(T_1,..., T_d)$.  
We refer the reader to \cite{KPTT} for the definitions of these properties for operator systems. 
Here we revisit the results of \cite{PaPa} in a basis free manner that allows us to study infinite dimensional versions 
and relates these properties to our constants $d_k(\cl S)$ and $r_k(\cl S).$
We prove the following theorem.

\begin{thm} \label{theorem:infinity} Let $\cl S$ be a finite dimensional operator system.
\begin{enumerate}
\item $\cl S$ has the lifting property if and only if $d_{\infty}(\cl S)=1$.
\item $\cl S$ is exact if and only if $r_{\infty}(\cl S)=1$.
\end{enumerate}
\end{thm}

We also show that an infinite dimensional operator system $\cl S$ satisfying $r_{\infty}(\cl S)=1$ is exact, and give an example showing that the converse does not hold.

Given a $d$-tuple $\mathrm{T}= (T_1,..., T_d)$ of elements of $B(\cl H)$ we let 
\[
 \cl S_{\rm T} = \mathspan \{ I, T_1,...,T_d, T_1^*,...,T_d^* \}.
\]
Given $k \geq 1$, let $T_i^{k{\kmin}}$ and $T_i^{k\kmax}$ denote the images of $T_i$ in the operator systems $\OMIN_k(\cl S_{\rm T})$ and $\OMAX_k(\cl S_{\rm T})$, respectively. 
Denote the corresponding $d$-tuples by ${\rm T}^{k\kmin}$ and ${\rm T}^{k\kmax}$, respectively.

The \emph{$n$-th matrix range} of $\rm T$ is the set of $d$-tuples of $n \times n$ matrices of the form
\[
 W^n({\rm T}) := \{ ( \phi(T_1), ..., \phi(T_d)) \vert \,\, \phi: \cl S_{\rm T} \to \M_n, \,\UCP \}.
\]
The \emph{matrix range} of $\rm T$ is the set $\cl W({\rm T}) = \{ W^n({\rm T}) : n \in \bb N \}$.
For those familiar with the concepts, these are the prototypical bounded matrix convex structures on $\bb C^d$.
It is not hard to see that
\[
 W^n( {\rm T}^{k\kmin}) \!\subseteq\! W^n({\rm T}^{(k+1)\kmin}) \!\subseteq\! W^n( {\rm T}) 
 \!\subseteq\! W^n({\rm T}^{(k+1)\kmax}) \!\subseteq\! W^n({\rm T}^{k\kmax}).
\]

Given subsets $X$, $Y$ in a metric space, their Hausdorff distance is
\[ 
 d_H(X,Y) = \max \Big\{ \sup_{y \in Y} \inf_{x \in X} d(x,y) , \sup_{x \in X} \inf_{y \in Y} d(x,y) \Big\}.
\] 
We define a metric on $d$-tuples of $n \times n$ matrices by setting 
\[
 d\big( (X_1,...,X_d), (Y_1,...,Y_d) \big) = \max_j \|X_j - Y_j \| .  
\]
Given matrix ranges $\cl W(\rm T) \subseteq \cl W(\rm S)$, define 
\[
 d_H( \cl W({\rm T}), \cl W({\rm S})) = \sup_{n\geq 1} d_H(W^n({\rm T}), W^n({\rm S})) . 
\]

We require the following result of Passer and the fourth author \cite{PaPa}.

\begin{thm}[\cite{PaPa}] \label{theorem:PaPa} Let ${\rm T}= (T_1,..., T_d)$ be a d-tuple of operators. 
\begin{enumerate}
\item $\cl S_{\rm T}$ has the LP if and only if $\displaystyle\lim_{k\to\infty} d_H(\cl W({\rm T}), \cl W({\rm T}^{k\kmax})) =0$.
\item $\cl S_{\rm T}$ is exact if and only if $\displaystyle\lim_{k\to\infty} d_H(\cl W({\rm T}^{k\kmin}), \cl W(\rm T)) =0$.
\end{enumerate}
\end{thm}

%
%

\begin{proof}[Proof of Theorem \ref{theorem:infinity}]
We deduce Theorem \ref{theorem:infinity} from Theorem \ref{theorem:PaPa}.
Let $\cl S$ be a finite-dimensional operator system and $T_0 = I, T_1,..., T_d$ be a self-adjoint basis for $\cl S$; so $\cl S = \cl S_{\rm T}$ for the $d$-tuple ${\rm T} = (T_1,..., T_d)$. It suffices to show that $d_\infty(\cl S)=1$ if and only $d_H(\cl W({\rm T}), \cl W({\rm T}^{k\kmax}))$ tends to $0$ as $k$ tends to infinity.

Let $\gamma= \id_{\cl S}^{\, \OMAX_k(\cl S)}$ be the map $\id: \cl S \to \OMAX_k(\cl S)$.
We begin by showing that 
\[
 d_H( \cl W( {\rm T}), \cl W({\rm T}^{k\kmax})) \leq 2 (\| \gamma\|_{cb}-1) .
\]

Fix $(A_1, ..., A_d) \in W^n({\rm T}^{k\kmax})$ and $\phi \in \UCP(\cl S_{{\rm T}^{k\kmax}}, \M_n)$ such that $\phi(T_i^{k\kmax}) = A_i$.
Then the map $\phi \circ \gamma: \cl S_{\rm{T}} \to \M_n$ is self-adjoint.
Hence by Wittstock's decomposition theorem it may be written as the difference of two CP maps $\phi_i \in \CP(\cl S_{{\rm T}}, \M_n)$,
$\phi \circ \gamma = \phi_1 - \phi_2$ with 
\[ \|\phi_1(I) + \phi_2(I)\| = \| \phi \circ \gamma\|_{cb} \leq \|\gamma\|_{cb} .\]

We have 
\[ I = \phi \circ \gamma (I)= \phi_1(I) - \phi_2(I) \leq \phi_1(I) \leq \phi_1(I) + \phi_2(I) \leq \|\gamma\|_{cb} I .\]
This implies that $\|\phi_2(I)\| \leq \|\gamma \|_{cb} -1$.
Let $P= \phi_1(I)$; so \mbox{$I \leq P \leq \|\gamma\|_{cb} I$}, whence $\|\gamma\|_{cb}^{-1/2} \leq \|P^{-1/2}\| \leq 1$. 
Define a map $\psi: \cl S_{\rm{T}} \to \M_n$  by $\psi(X) = P^{-1/2}\phi_1(X) P^{-1/2}$. 
Then $\psi$ is UCP and hence, 
\[ (B_1,...,B_d) =: (\psi(T_1),..., \psi(T_d)) \in W^n(\rm{T}) .\]

For a real number $x \geq 0$, we have
\[
 1- \frac 1{\sqrt{1+x}} = \frac x {\sqrt{1+x}(\sqrt{1+x} + 1)} \leq \frac x 2 
\]
and therefore
\begin{align*} 
 \|A_i - B_i\| &= \|A_i - P^{-1/2}( \phi \circ \gamma(T_i) + \phi_2(T_i) ) P^{-1/2} \| \\&
 \leq  \| A_i - P^{-1/2} A_i P^{-1/2} \| + \|P^{-1/2} \phi_2(T_i) P^{-1/2}\|  \\&
 \leq 2 \| I - P^{-1/2} \| + \|P^{-1} \| (\|\gamma\|_{cb} -1) \|  \\&
 \leq 2 \Big( 1- \frac{1}{\sqrt{\|\gamma\|_{cb}}} \Big) + (\|\gamma\|_{cb} -1) \leq 2(\| \gamma \|_{cb} -1). 
\end{align*}
Thus, if the cb-norms tend to 1, then the Hausdorff distance tends to 0.

For the converse, we use the elementary fact that if $T_i \in \cl T$ and matrices $A_i \in \M_n, 0 \leq i \leq d$, then
\[
 \big\| \sum_{i=0}^d A_i \otimes T_i \big\| = \sup \big\{ \big\| \sum_i A_i \otimes B_i \big\| : (B_0,\dots, B_d) \in \cl W(\rm T) \big\}.
\]

The map from $\e: \cl S \to \ell^{\infty}_{d+1}$ given by $T_i \to e_i$ is completely bounded.
Thus if $A_i \in \M_n, 0 \leq i \leq d$, then
\[
 \max_{0 \leq i \leq d} \|A_i\| = \big\| \e \big( \sum_{i=0}^d A_i \otimes T_i \big) \big\|  \leq \| \e \|_{cb} \big\| \sum_{i=0}^d A_i \otimes T_i \big\| .
\]
Given any ${\rm B} \in W^n({\rm T}^{k\kmax})$, choose ${\rm C} \in W^n(\rm T)$ such that 
\[ \max_i \|B_i - C_i \| \leq d_H(\cl W({\rm T}), \cl W({\rm T}^{k\kmax})) .  \]
We have that
\begin{align*}
  \| \sum_i A_i \otimes B_i \| &\leq   \| \sum_i A_i \otimes C_i \| + \| \sum_i A_i \otimes (B_i - C_i) \| \\&
  \leq \| \sum_i A_i \otimes T_i \| + (\max_i \|A_i \|) d_H(\cl W({\rm T}), \cl W({\rm T}^{k\kmax})) \\&
  \leq \big(1 + \|\e \|_{cb}\, d_H(\cl W({\rm T}), \cl W({\rm T}^{k\kmax})) \big) \|\sum_i A_i \otimes T_i \|.
\end{align*}
Taking the supremum of the left hand side over all ${\rm B} \in \cl W({\rm T}^{k\kmax})$ and over all $A_i$, yields
\[
 \|\gamma \|_{cb} \leq 1 + \|\e \|_{cb}\, d_H(\cl W \big( {\rm T}), \cl W({\rm T}^{k\kmax}) \big) .
\]
Since $\|\e\|_{cb}$ is independent of $k$, we see that if the Hausdorff distance tends to 0, then the cb-norm tends to 1.

The proof for (2) is similar.
\end{proof}

One implication of the exactness result can be extended to infinite dimensions. For this we shall use the tensor product characterization of exactness from \cite{KPTT}. To prove this it is convenient to first prove a preliminary result whose C*-analogue is well-known.

\begin{prop} Let $\cl S$ be an operator system that has a unital complete order inclusion into a nuclear C*-algebra. Then $\cl S$ is exact.
\end{prop}
\begin{proof} Let $\cl A$ be the nuclear C*-algebra that contains $\cl S$.  By the universal properties of the injective envelope there exists a UCP map  $\psi: \cl A \to I(\cl S)$ such that $\psi|_{\cl S}$ is the identity map.

We use the subscript $coi$ to indicate that an inclusion or equality of sets is a unital complete order inclusion.

Given any operator system $\cl T$ we have that
\[
 \cl S \otimes_{min} \cl T \subseteq_{coi} \cl A \otimes_{min} \cl T =_{coi} \cl A \otimes_{max} \cl T 
 \stackrel{\psi \otimes id}{\longrightarrow} I(\cl S) \otimes_{max} \cl T.
\]

Thus, the inclusion  $\cl S \otimes \cl T \subseteq I(\cl S) \otimes_{max} \cl T$ defines the min tensor, but this subspace identification is also how the el-tensor is defined.  

Hence, $\cl S \otimes_{min} \cl T =_{coi} \cl S \otimes_{el} \cl T$ for every $\cl T$ and by \cite{KPTT}*{Theorem~5.7}, $\cl S$ is exact.
\end{proof}

\begin{prop} 
Let $\cl S$ be an operator system, then $\OMIN_k(\cl S)$ is an exact operator system. 
\end{prop}

\begin{proof} 
Let $\Omega_k:= \UCP(\cl S, \M_k)$ equipped with the weak*-topology.  Then the map
\[
 x \in \OMIN_k(\cl S) \mapsto \hat{x} \in \M_k(C(\Omega_k)), \text{ where } \hat{x}(\phi) = \phi(x),
\]
is a unital complete order embedding. 
Since $\M_k(C(\Omega_k))$ is a nuclear C*-algebra,  $\OMIN_k(\cl S)$ is an exact operator system by the above result.
\end{proof}

\begin{lemma} 
Let $\gamma: \cl S_1 \to \cl S_2$ be self-adjoint, i.e., $\gamma(x^*) = \gamma(x)^*$, and cb, then for any operator system $\cl T$ we have that
\[
 \| \gamma \otimes \id: \cl S_1 \otimes_{el} \cl T \to \cl S_2 \otimes_{el} \cl T \|_{cb} = \|\gamma \|_{cb}.
\]
\end{lemma}

\begin{proof}
Regard $\gamma$ as a map into $I(\cl S_2)$, the injective envelope of $\cl S_2$.
Then by Wittstock's decomposition theorem and injectivity of the range, there exist CP maps $\gamma_i: \cl S_1 \to I(\cl S_2)$
such that $\gamma = \gamma_1 - \gamma_2$ and $\| \gamma_1(1) + \gamma_2(1) \| = \| \gamma \|_{cb}$.
Then $\gamma \otimes \id = \gamma_1 \otimes \id - \gamma_2 \otimes \id$, where
$\gamma_i \otimes \id: \cl S_1 \otimes_{el} \cl T \to I(\cl S_2) \otimes_{max} \cl T$ are both CP.
Hence, 
\begin{align*}
 \|\gamma \otimes \id \|_{CB(\cl S_1 \otimes_{el} \cl T, \cl S_2 \otimes_{el} \cl T)} 
 &= \| \gamma \otimes \id \|_{CB(\cl S_1 \otimes_{el} \cl T, I(S_2) \otimes_{max} \cl T)} \\
 &\leq \|(\gamma_1 +\gamma_2)\otimes \id\|_{CB(\cl S_1 \otimes_{el} \cl T, I(\cl S_2) \otimes_{max} \cl T)} \\
 &= \| \gamma_1(1) + \gamma_2(1) \| = \| \gamma \|_{cb} .
 \qedhere
\end{align*} 
\end{proof}

\begin{thm} Let $\cl S$ be an operator system.
If $r_{\infty}(\cl S) =1$, then $\cl S$ is exact. 
\end{thm}

\begin{proof}
Let $\alpha_k=\id: \cl S \to \OMIN_k(\cl S)$ be the identity map from $\cl S$ to $\OMIN_k(\cl S)$ and note that from the universal property (see the beginning of section \ref{sec: background}) of $\OMIN_k(\cl S)$, $\alpha_k$ is a unital completely positive map.  
Let also $\beta_k = \id_{\OMIN_k(\cl S)}^{\ \cl S} = \id : \OMIN_k(\cl S) \to \cl S$. Thus we have the following commutative diagram:
\[
 \begin{CD} \cl S \otimes_{min} \cl T & @>>> & \cl S \otimes_{el} \cl T \\
 @V\alpha_k\otimes \id VV & & @A\beta_k\otimes \id AA \\
 \OMIN_k(\cl S) \otimes_{min} \cl T  & @= & \OMIN_k(\cl S) \otimes_{el} \cl T.
\end{CD}
\] 
Since by the lemma,  $\| \beta_k \otimes \id \|_{cb} = \| \beta_k \|_{cb} \to 1$, we have that the top arrow is unital and completely contractive and hence UCP.
So $\cl S$ is exact by \cite{KPTT}*{Theorem~5.7}.
\end{proof}

We give an example of an exact infinite dimensional operator system for which this limit is not 1. In fact, our example is a C*-algebra.

\begin{example}
Let $\cl A = \cl K(\ell^2(\bb N)) + \bb C I_{\ell^2(\bb N)}$ denote the unitized compact operators.  This C*-algebra is nuclear and hence exact.

Let $P_n$ denote the projection onto the first $n$ coordinates, let $\phi_n: \cl A \to \M_n$ by given by $\phi_n(X) = P_nXP_n$ denote the compression to the first $n \times n$ block, and let $\psi_n: \M_n \to \cl A$ be the map given by
\[
 \psi_n(Y) = Y \oplus tr_n(Y) (I-P_n).
\]  
Then both of these maps are UCP and for every $X \in \cl A$ we have that
\[
 \|X - \psi_n \circ \phi_n(X) \| \to 0.
\] 
In fact, this is a way that one can directly prove that $\cl A$ is a nuclear operator system.

Next we show that $\psi_{n,k} = \id_{\cl A}^{\,\OMIN_k(\cl A)}\circ \psi_n \circ \id_{\OMIN_k(\M_n)}^{\ \M_n}$ is UCP.  
Let  $(Y_{i,j}) \in \M_p(\OMIN_k(\M_n))^+$ and let $\delta: \cl A \to \M_k$ be UCP. Then
\[
 ( \delta \circ \psi_{n,k}(Y_{i,j})) = (\delta(Y_{i,j} \oplus 0)) + (\delta( 0 \oplus tr_n(Y_{i,j}) I)).
\]
If we define $\delta_0: \M_n \to \M_k$ via $\delta_0(Y) = \delta(Y \oplus 0)$, then this is a CP map into $\M_k$ 
and so the image of an element of $\M_p(\OMIN_k(\M_n))^+$ will be positive. 
Similarly, $\delta_1: \M_n \to \M_k$ defined by $\delta_1(Y) = \delta(0 \oplus tr_n(Y)(I-P_n)) I_k$ is CP into $\M_k$ 
and so $(\delta_1(Y_{i,j})) \in \M_p(\M_k)^+$.
Thus, the image under $\delta$ of each positive element of $\M_p(\OMIN_k(\M_n))^+$ is positive in $\M_p(\M_k)$. 
Since this is true for every CP map into $\M_k$, it follows from \cite{JoKrPaPe}*{Theorem~5} that 
$(\psi_{n,k}(Y_{i,j})) \in \M_p(\OMIN_k(\cl A))^+$; whence $\psi_{n,k}$ is CP.

Let $\beta_k = \id_{\OMIN_k(\cl A)}^{\ \cl A}$ denote the identity map of $\OMIN_k(\cl A)$ into $\cl A$.  
Let $\gamma_{n,k} = \id_{\OMIN_k(\M_n)}^{\,\M_n}$ denote the identity map of $\OMIN_k(\M_n)$ into $\M_n$. 
Note that $\gamma_{n,k} = \phi_n \circ \beta_k \circ \psi_{n,k}$.
Hence, for all $n \geq k$,
\[
 \frac{2n-k}{k} = \|\gamma_{n,k} \|_{cb} \leq \|\phi_n\|_{cb} \| \beta_k \|_{cb} \|\psi_{n,k} \|_{cb} = \|\beta_k \|_{cb},
\]
from which it follows that $ \|\id_{\OMIN_k(\cl A)}^{\ \cl A} \|_{cb} = \|\beta_k \|_{cb} = + \infty$.
Thus, $\cl A$ is exact but 
\[ r_{\infty}(\cl A) :=
 \lim_{k\to\infty} \|\id_{\OMIN_k(\cl A)}^{\ \cl A} \|_{cb} = + \infty .
\] 
\end{example}

\begin{remark} There are many finite dimensional operator systems that are known to be not exact or not have the lifting property.  The values of the parameters $d_{\infty}(\cl S)$ and $r_{\infty}(\cl S)$ are not known for most such examples. Even in the case that $\cl S$ is a concrete operator subsystem of a matrix algebra. For example, if $\bb F_n$ denotes the free group on $n$ generators, then the $2n+1$ dimensional operator system $\cl S_n \subseteq C^*(\bb F_n)$ spanned by the generators of the free group, their adjoints, and the identity, has the lifting property but is not exact. Thus, $r_{\infty}(\cl S_n) \ne 1$ but its exact value is unknown as are the values of $r_k(\cl S_n)$. Similarly, while $d_{\infty}(\cl S_n) =1$, values of $d_k(\cl S_n)$ are not known. 

   The 5 dimensional operator subsystem of $\M_4$ given by
\[ \cl S:= \operatorname{span} \{ I_4, E_{1,2}, E_{3,4}, E_{2,1}, E_{4,3} \},\]
is known to fail to have the lifting property \cite{KPTT}, but the values of $d_k(\cl S)$ and $d_{\infty}(\cl S)$ are not known.
\end{remark}

\section{Acknowledgements}

We thank the referee for helpful comments, especially bringing \cite{RS89} to our attention.

GA was supported in part by ANR (France) under the grant
ESQuisses (ANR-20-CE47-0014-01). AMH acknowledges funding from The Research Council
of Norway (project 324944). MR is supported by the Marie Sk\l{}odowska-Curie Fellowship from the European Union’s Horizon research and innovation programme, grant Agreement No.\ HORIZON-MSCA-2022-PF-01 (Project number: 101108117).

 \appendix

 \section{Partially entanglement breaking maps}

 We will use some results about \emph{$k$-partially entanglement breaking} ($k$-PEB) maps. 
Recall that a map $\psi: \M_n \to \M_m$ is called $k$-PEB provided $\psi \circ \phi$ is CP for every $k$-positive map $\phi: \M_d \to \M_n$, with $d$ arbitrary. These maps are also called \emph{$k$-superpositive} at various places in the literature (see \cite{super-positive} and the references therein).

\begin{prop} Let $\psi: \M_n \to \M_m$ be $k$-PEB. 
If $\cl T$ is any operator system and $\phi: \cl T \to \M_n$ is $k$-positive, then $\psi \circ \phi$ is completely positive.
\end{prop}

\begin{proof} 
Let $X = (x_{i,j}) \in \M_d(\cl T)^+ = (M_d \otimes \cl T)^+$ for $d\geq 1$. 
Then there is a completely positive map $\gamma: \M_d \to \cl T$ with $\gamma(E_{i,j}) = x_{i,j}$, 
where $E_{i,j}$ denotes the standard matrix unit basis for $\M_d$.  
Indeed, $\gamma(A) = V^*(A \otimes X) V$ where $V = \sum_{i=1}^d E_{i1} \otimes (E_{i1} \otimes 1_{\cl T})$.
Then the map $\phi \circ \gamma: \M_d \to \M_n$ is $k$-positive and hence $\psi \circ \phi \circ \gamma$ is completely positive.

Let $E = (E_{i,j})) \in \M_d(\M_d)^+$.
We see that $(\psi \circ \phi)(X) = ( \psi \circ \phi \circ \gamma)(E)$ is positive. 
Since $d$ is arbitrary, this shows that $\psi \circ \phi$ is completely positive.
\end{proof}


 Another way of stating these results is given by the following proposition:

 \begin{prop} \label{k-PEBprop} 
Given a map $\psi: \M_n \to \M_m$, the following are equivalent:
\begin{enumerate}
 \item $\psi:\OMIN_k(\M_n) \to \M_m$ is CP,
 \item $\psi: \M_n \to \OMAX_k(\M_m)$ is CP,
 \item $\psi$ is $k$-PEB.
\end{enumerate}
\end{prop}

\begin{proof} 
Assume (1). If $\phi: \M_d \to \M_n$ is $k$-positive, then $\phi: \M_d \to \OMIN_k(\M_n)$ is CP and hence $\psi \circ \phi$ is CP.  
Therefore, $\psi$ is $k$-PEB.
 
Conversely, if $\psi$ is $k$-PEB, then since $\phi=\id: \OMIN_k(\M_n) \to \M_n$ is $k$-positive, we have that
$\psi \circ \phi: \OMIN_k(\M_n) \to \M_n$ is CP. 
Since $\psi \circ \phi(X) = \psi(X)$ we have (1).
Thus, (1) and (3) are equivalent.

The equivalence of (2) and (3) is similar.
\end{proof}


\end{document}